\documentclass[11pt,centertags]{amsart}
\usepackage{amssymb}            
\usepackage{mathrsfs}             
\usepackage{hyperref}           
\usepackage{calc}               
\usepackage{enumerate}          
\usepackage{fancyhdr}           
\usepackage{url}                
\usepackage{gloss}              
\usepackage{yhmath}             
\usepackage{graphicx}

\newtheorem*{main*}{Main Theorem}

\newtheorem{theorem}{Theorem}[section]
\newtheorem*{theorem*}{Theorem}

\newtheorem{lemma}[theorem]{Lemma}
\newtheorem*{corollary}{Corollary}

\newtheorem*{question*}{Question}

\newtheorem*{conjecture*}{Conjecture}

\theoremstyle{definition}

\newtheorem*{definition*}{Definition}

\theoremstyle{remark}
\newtheorem{remark}[theorem]{Remark}

\numberwithin{equation}{section}

\newcommand{\R}{\mathbb{R}}

\newcommand{\Z}{\mathbb{Z}}

\newcommand{\mc}{\mathcal}

\newcommand{\Ga}{\Gamma}

\newcommand{\La}{\Lambda}

\DeclareMathOperator{\cdim}{cdim}

\DeclareMathOperator{\Aff}{Aff}
\DeclareMathOperator{\GL}{GL}

\DeclareMathOperator{\Sp}{\Sp}

\newcommand{\til}{\widetilde}

\providecommand{\to}{\longrightarrow }

\newcommand{\norm}[1]{\left\| #1 \right\| }

		\renewcommand{\bf}{\bfseries}

		\renewcommand{\it}{\itshape}

		\renewcommand{\hat}{\widehat}
		\renewcommand{\bar}{\overline}
		
\begin{document}

\author[Michelle Bucher]{Michelle Bucher}
\author[Chris Connell]{Chris Connell$^\dagger$}
\thanks{$\dagger$ This work is supported by the Simons Foundation, under grant \#210442}
\author[Jean-Fran\c{c}ois Lafont]{Jean-Fran\c{c}ois Lafont$^\ddagger$}
\thanks{$\ddagger$ This work is supported by the NSF, under grant DMS-1510640.}

\title{Vanishing simplicial volume for certain affine manifolds}

\address{Universit\'e de Gen\`eve}
\email{Michelle.Bucher-Karlsson@unige.ch}

\address{Indiana University} 
\email{cconnell@indiana.edu} 

\address{Ohio State University} 
\email{jlafont@math.ohio-state.edu} 

\begin{abstract}
We show that closed aspherical manifolds supporting an affine structure, whose holonomy map is injective and 
contains a pure translation, must have vanishing
simplicial volume. This provides some further evidence for the veracity of the Auslander Conjecture. 
Along the way, we provide a simple cohomological criterion for aspherical manifolds with normal 
amenable subgroups of $\pi_1$ to have vanishing simplicial volume. This answers a special case of a question
due to L\"uck. 

\end{abstract}

\maketitle


\thispagestyle{empty} 

\section[]{Introduction}

The topology of affine manifolds remains quite poorly understood. In this short note, we consider the simplicial volume 
of affine manifolds. We show:

\begin{main*}\label{thm: affine simplicial volume} Let $M$ be a closed aspherical manifold. Suppose that $M$ admits an 
affine structure for which the holonomy representation $\rho:\pi_1(M)\rightarrow \mathrm{Aff}(\mathbb{R}^n) = \R^n \rtimes GL_n(\mathbb R)$ 
is injective, and has non-trivial translational subgroup $\rho(\pi_1(M))\cap \mathbb R^n$. 
Then the simplicial volume of $M$ vanishes.
\end{main*}

Recall that, for a closed oriented manifold $M$, the simplicial volume $||M||$ is a topological invariant which measures 
how efficiently the fundamental class of $M$ can be represented as a real singular chain. This non-negative real valued
invariant was introduced by Gromov \cite{Gr82} and Thurston \cite[Chapter 6]{thurston}. 

A smooth manifold $M$ supports an affine structure if one can choose charts for $M$ so that all transition maps are affine maps.
If $M$ has an affine structure, then there is an associated {\it holonomy
representation} $\rho: \pi_1(M) \rightarrow \mathrm{Aff}(\mathbb{R}^n)$, and a $\rho$-equivariant
{\it developing map} $D: \tilde M \rightarrow \mathbb R^n$ (unique up to affine transformations). If the developing map $D$ is a 
homeomorphism, then the affine structure is called {\it complete}. 

If $M$ has a complete affine structure, then it 
follows that the holonomy representation is injective, and that $M$ is aspherical. 
Thus from our {\bf Main Theorem} we obtain (see also the discussion in Section \ref{complete-case}):

\begin{corollary}
If $M$ is a closed manifold with a complete affine structure, and the holonomy contains a pure translation, then $||M|| = 0$.
\end{corollary}

On the other hand, there exist examples of complete affine manifolds whose (necessarily injective) holonomy representation
contains no pure translations -- and hence are not covered by our {\bf Main Theorem}. 

We have established a special case of the following natural:

\begin{conjecture*}
If $M$ is a closed manifold supporting an affine structure, then $||M||=0$.
\end{conjecture*}

\begin{remark}
In the context of closed complete affine manifolds, the Auslander Conjecture predicts that the fundamental group of such a 
manifold is virtually polycyclic.
Since manifolds whose fundamental group are virtually polycyclic have vanishing simplicial volume, for this class of manifolds our conjecture would follow immediately from the Auslander Conjecture. In particular, from the work of 
Abels-Margulis-Soifer \cite{Abels-Margulis-Soifer}, we see that for {\it complete} affine manifolds,
our conjecture holds in dimensions $\leq 6$.
\end{remark}

\begin{remark}
Another famous problem is the Chern Conjecture, which asserts that affine manifolds have zero Euler characteristic.
But a well-known conjecture of Gromov predicts that aspherical manifolds with vanishing simplicial volume 
automatically have vanishing Euler characteristic. So if Gromov's conjecture is correct, then our conjecture would immediately imply the Chern 
Conjecture. The Chern Conjecture has so far only been established for particular families of affine structures or of manifolds. For example, it is 
known to hold for complete affine manifolds \cite{KoSu75}, for affine manifolds with linear holonomy in $\mathrm{SL}(n,\mathbb{R})$ \cite{Kl15} 
(a conjecture of Markus predicts this is equivalent to being complete), for surfaces \cite{Be60}, for higher rank irreducible locally symmetric 
manifolds \cite{GoHi84}, for manifolds which are locally a product of hyperbolic planes \cite{BuGe11}, and 
for complex hyperbolic surfaces \cite{Pi16}. 
\end{remark}

\vskip 5pt

\centerline{\bf Acknowledgments}

\vskip 5pt
C.C. and J.-F.L. 
would like to thank the SNFS for its support (grant 200020-159581/1) and the University of Geneva for its hospitality during a visit when much of
this work was completed. The authors are grateful to the American Institute of Mathematics for funding a SQuaRE program 
``New Horizons for Simplicial Volume and the Barycenter Method'', where several ideas of the present paper originated. 

The authors would like to thank Jim Davis, Mike Davis, Bill Goldman, 
and Mike Mandell for helpful comments. We are particularly indebted to Clara L\"oh, for many insightful comments, and for
catching a gap in an earlier version of this paper. We would also like to thank Christoforos Neofytidis, for drawing our attention to 
Fel'dman's paper \cite{F71}.

\section{Normal amenable subgroups and a question of L\"uck}\label{luck}

In L\"uck's book on $L^2$-invariants, the following question is raised -- see \cite[Question 14.39]{L02}:

\begin{question*} Let $M$ be a closed aspherical manifold whose fundamental group contains 
a non trivial amenable normal subgroup $A\triangleleft\pi_1(M)$. Does the simplicial volume of $M$ vanish?
\end{question*}

Affirmative answers to this question are only known in some special cases. It is easy to check for fibrations for which the fiber has amenable fundamental group \cite[Exercise 14.15]{L02}. Furthermore, Neofytidis proves it for aspherical manifolds whose fundamental group is infinite index presentable by products 
while the quotient by the center of the fundamental group is not presentable by products \cite[Corollary 1.2]{Ne15}.
While we will not require it for the proof of our {\bf Main Theorem}, we first establish a special case of the question. Both of 
these results will rely on the following elementary lemma.

\begin{lemma}\label{lem:claim1}
Let $M$ be a closed connected $n$-manifold, and $\Ga = \pi_1(M)$. Assume $A\triangleleft \Ga$ is an amenable
normal subgroup, and $q:\Ga \rightarrow \La := \Ga/A$ is the quotient map.
If the induced map $q^*:H^n(\La ; \R)\to H^n(\Ga ; \R)$ is the zero map, then $\norm{M}=0$.
\end{lemma}

\begin{proof}
We have the 
following commutative diagram:
\begin{equation*}
  \label{eq:cd1}
  \raisebox{-0.5\height}{\includegraphics[scale=1.1]{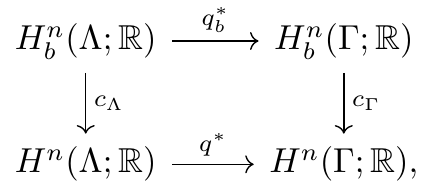}}
\end{equation*}
where the vertical arrows are the comparison maps from bounded cohomology to ordinary cohomology, while the
horizontal arrows are the morphisms induced by the surjection $q: \Ga \twoheadrightarrow \La$. 
From the hypothesis, the bottom map $q^*$ is the zero map. On the other hand, 
since $A$ is amenable, we have that the top map $q^*_b$ is an isomorphism 
(see \cite[Section 3.1]{Gr82}, or take $E=\R$ in \cite[Remark 8.5.4]{Monod}). Commutativity of the diagram
forces $c_\Ga$ to also be the zero map. 

Next, consider the classifying map $\phi:M\rightarrow B\Ga$. The induced map on bounded cohomology is always
an isomorphism, so we get a commutative diagram:
\begin{equation*}
  \label{eq:cd2}
  \raisebox{-0.5\height}{\includegraphics[scale=1.1]{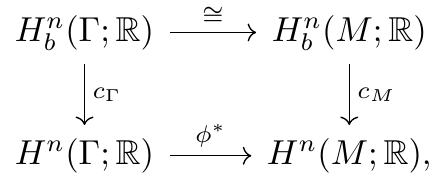}}
\end{equation*}
Since $c_\Ga$ is the zero map, so is $c_M$, which immediately implies $||M||=0$ by the ``duality'' of $\ell^1$ and $\ell^\infty$-norms \cite{Gr82}.
\end{proof}

\begin{theorem}\label{thm:main} Let $M$ be a closed connected aspherical $n$-manifold whose fundamental group 
$\Ga=\pi_1(M)$ contains a non-trivial amenable normal subgroup $A\triangleleft\pi_1(M)$, and let $\Lambda=\pi_1(M)/A$. 
Assume that the quotient group $\La$ has {\em finite} cohomological dimension $\cdim _{\mathbb R}(\Lambda) = \ell <\infty$,
and that $H^\ell(\Lambda ; H^k(A ; \mathbb{R}))\neq 0$, where $k=\cdim _{\mathbb R}(A)$.  Then $||M|| = 0$.
\end{theorem}

\begin{proof} 
We consider the cohomological Lyndon-Hochschild-Serre spectral sequence with real coefficients $\R$ (and
trivial module structure) associated to the short exact sequence $1\rightarrow A \rightarrow \Ga \rightarrow \La\rightarrow 1$.  
The $E_2$-page is given by
\[
E_2^{p,q}:=H^p(\La ; H^q( A ; \R)) \Rightarrow H^{p+q}(\Ga ; \R),
\]
and the spectral sequence convergences to the cohomology $H^*(\Ga ; \R)\cong H^*(M ; \R)$. 
Note that $\tilde M/A$ is an $n$-dimensional model for a $K(A,1)$, and hence $\cdim_{\R}(A) = k \leq n$ is finite. 
From our hypotheses, we obtain the following observations:
\begin{enumerate}
\item Since $\cdim_{\R}(A)=k$, it follows that $H^q(A ; \R)=0$ for all $q>k$. This
forces $E_2^{p,q}=0$ for all $q>k$.
\item Since $\cdim_{\R}(\La)=\ell$, we see that $H^p(\La ; -)=0$ for all $p>\ell$, regardless of the coefficient 
$\mathbb R[\Lambda]$-module. 
In particular, this forces $E_2^{p,q}=0$ for all $p>\ell$.
\end{enumerate}
Whence we see that the $E^2$-page looks like
\begin{equation*}
  \label{eq:cd3}
  \raisebox{-0.5\height}{\includegraphics[scale=1.1]{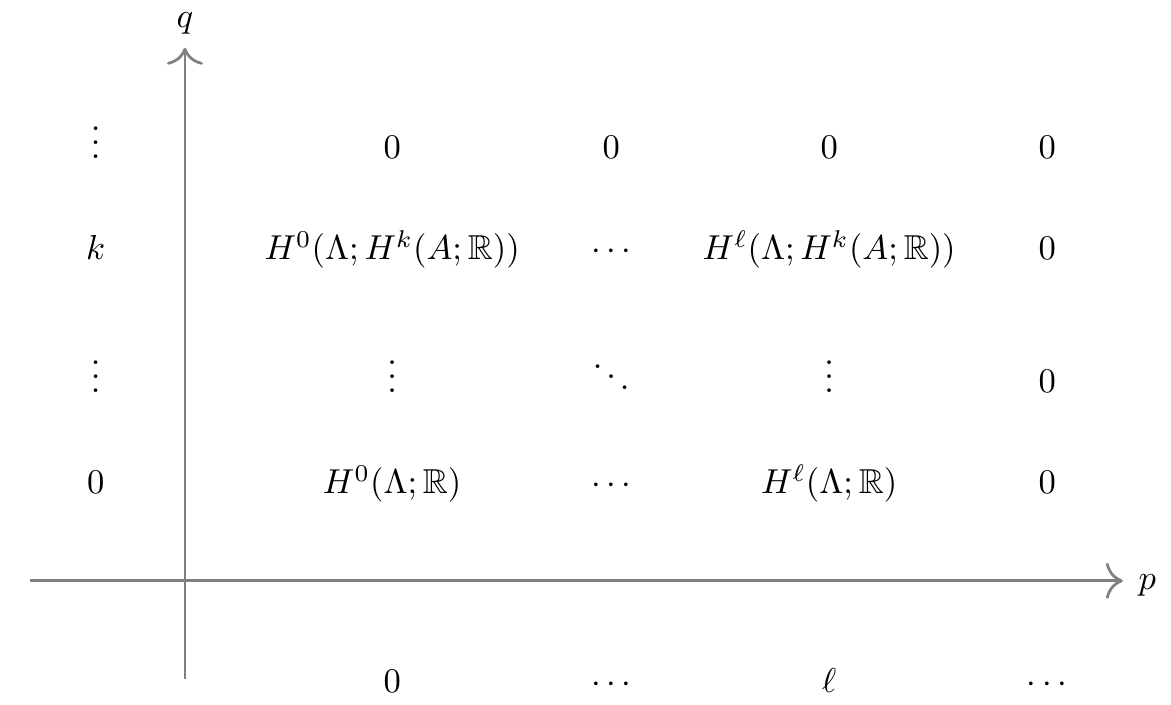}}
\end{equation*}

By hypothesis, we also have that the $E_2^{\ell ,k}$ entry is non-zero. Thus the $E_2^{\ell ,k}$
entry survives to the $E_\infty$-page, establishing that $\cdim_{\mathbb R}(\Ga) \geq \ell +k$.

Now the closed orientable aspherical manifold $M^n$ is a model for $K(\Ga, 1)$, so we obtain the lower bound 
$n\geq \ell + k$. Since $A$ is non-trivial, we have that $k>0$, and hence that $n>\ell=\cdim_{\mathbb R}(\La)$.
This forces $H^n(\La ; \R)=0$, and Lemma \ref{lem:claim1} allows us to conclude $\norm{M}=0$.
\end{proof}

\begin{remark}
As was pointed out to us by C. L\"oh, the proof of Theorem \ref{thm:main} still works if instead of
$M$ aspherical, we only have $\cdim _\R(\pi_1(M))=n$.
\end{remark}

\begin{remark} It is tempting to use Lemma \ref{lem:claim1} to attack the general case of L\"uck's question. Notice
that the induced homomorphism $q^*: H^n(\La; \R) \rightarrow H^n(\Ga ; \R)$ appears naturally inside
the Lyndon-Hochschild-Serre spectral sequence. Indeed, $H^n(\La; \R)$ appears as the $E_2^{0,n}$-term in the spectral
sequence. Thus, whether or not the induced homomorphism $q^*$ is zero translates to whether or not the $E_2^{0,n}$ survives
to the $E_\infty$-page, i.e. whether or not $E_\infty^{0,n}=0$.  In the proof of Theorem \ref{thm:main}, our hypotheses already 
forced $E_2^{0,n} =0$.
\end{remark}


\section{Proof of Main Theorem}\label{proof-main-thm}

This section is devoted to the proof of the {\bf Main Theorem}. So let us assume that $M$ is a connected closed aspherical affine manifold, 
$\Ga = \pi_1(M)$, and the holonomy representation $\rho: \Ga \rightarrow \Aff(\R^n) = \R^n \rtimes GL_n(\R)$ is injective with 
$\rho(\Gamma)\cap \mathbb R^n$ non-zero. Since the simplicial volume is multiplicative under finite covers, it is sufficient to show that a finite 
cover of $M$ has vanishing 
simplicial volume. Since the hypotheses in our theorem are inherited by finite covers, we will from now on assume that the manifold
$M$ is orientable.

We have the following commutative diagram relating the various groups
we are interested in:
\begin{equation}\label{eq:Aff}
  \raisebox{-0.5\height}{\includegraphics[scale=1.1]{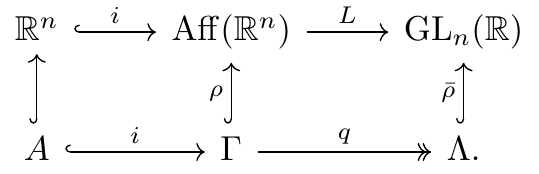}}
\end{equation}
Here $A$ is the purely translational part of $\pi_1(M)$ -- and by hypothesis, $A$ is non-trivial, so of
rank $\geq 1$. This forces $\cdim_{\mathbb R}(A)\geq 1$. 
Since $A\triangleleft \pi_1(M)$ and $M$ is aspherical, $\pi_1(M)$ and $A$ are torsion free. 
Note that $\tilde M /A$ is an $n$-dimensional model for a $K(A,1)$, which immediately gives us:

\vskip 5pt

\noindent {\bf Fact 1:} $\cdim_{\R}(A)=k$, with $1\leq k \leq n$, and hence $A\cong \Z ^k$.

\vskip 5pt

Next we consider $\cdim_{\R}(\La)$, where $\La$ is the linear part of the holonomy action. A special case of the main theorem 
of \cite{Alperin-Shalen} states:

\vskip 10pt

\begin{theorem*}[Alperin-Shalen]
If $S$ is a finitely generated integral domain of characteristic zero, then $G<\GL(S)$ has finite cohomological dimension
$\cdim_{\Z} (G)<\infty$ if and 
only if there is an upper bound on the ranks of abelian subgroups of $G$.
\end{theorem*}

\vskip 10pt

For a finite generating set $\{g_1, \ldots, g_r\}\subset \La$, take $S\subset \R$ to be the subring of $\R$ generated (over $\Z$) by the finite
collection of matrix entries of $\{\bar \rho(g_1), \ldots , \bar \rho(g_r)\}$. Then $S$ is a (finitely generated) characteristic zero integral domain, 
since it is a subring of $\R$, and $\bar \rho(\La) \subset \GL(S) \subset \GL_n(\R)$. We now use the embedding $\bar \rho$ to identify $\La$
with its isomorphic copy in $\GL(S)$. Since $\La$ is a finitely generated linear group, it has a finite index torsion-free subgroup $\La ^\prime$; 
we replace $\La, \Ga$ by the finite index subgroups $\La ^\prime, \Ga^\prime:=q^{-1}(\La ^\prime)$. This replaces $M$ by a finite cover $M^\prime$,
so we can now also assume that the quotient group $\La$ is torsion-free.

Taking a finitely generated abelian subgroup $H<\La$ (necessarily torsion-free), we have a corresponding exact sequence:
\begin{equation}\label{eq:H}
  \raisebox{-0.5\height}{\includegraphics[scale=1.1]{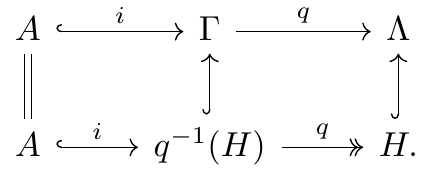}}
\end{equation}
Since $A$, $H$ are finitely generated torsion-free abelian, we see that $\hat H:= q^{-1}(H)$ is a 
finitely generated nilpotent group. Also, $\til{M}/\hat H$ is an $n$-dimensional $K(\hat H,1)$.  
Hence $h(\hat H)=\cdim_{\Z}(\hat H)\leq n$, where $h(\hat H)$ is the Hirsch length of $\hat H$ (see Gruenberg \cite[Section 8.8]{G70}). 
But from the two step 
nilpotence sequence above, the Hirsch length of $\hat H$ is just $k+r$ where $r$ is the rank of $H$. 
Hence $k+ r\leq n$, and $n-k$ is the desired 
upper bound on the rank of the abelian subgroups of $\La$. Applying Alperin and Shalen's result, we conclude that $\La$ 
has finite cohomological dimension $\cdim_{\Z}(\La) <\infty$. Since any finite length free $\Z [\La]$ resolution
of $\Z$ can be tensored with $\R$ to obtain a same length free $\R [\La]$ resolution of $\R$, this yields
$\cdim_{\R}(\La) \leq \cdim_{\Z}(\La)$, which establishes 

\vskip 5pt

\noindent {\bf Fact 2:} $\cdim_{\R}(\Lambda)=\ell$ for some finite $\ell$.

\vskip 5pt

Finally, we will use the following result of Fel'dman \cite[Theorem 2.4]{F71}

\vskip 10pt

\begin{theorem*}[Fel'dman]
For $G$ a group, and $H\triangleleft G$ a normal subgroup, $F$ a field. If $H$ is of type $FP$ (over the field $F$), and 
$\cdim_F(G/H)<\infty$, then 
$$\cdim _F (G)=\cdim _F(H) + \cdim _F(G/H).$$
\end{theorem*}

\vskip 10pt

From {\bf Fact 1} and {\bf Fact 2} we see that the hypotheses of Fel'dman's theorem hold for $A\triangleleft \Gamma$
(with $F=\R$).
Since $\cdim_\R (A)>0$, Fel'dman's theorem gives us the inequality $\cdim _\R(\La)< \cdim _\R(\Ga)=n$. 
Applying Lemma \ref{lem:claim1} concludes the proof of the {\bf Main Theorem}.

\begin{remark}
As the reader can easily see, this same proof applies to the following more general setting. Let $M$ be closed connected $n$-manifold,
and assume that $\Ga = \pi_1(M)$ has $\cdim_\R (\Ga) = n$. If $A \triangleleft \Ga$ is a normal elementary amenable subgroup of type
$FP$, and 
$\La =\Ga / A$ is a linear group, then $||M||=0$. For the portions of the proof relying on Hirsch length, one can use Hillman's extension
of the Hirsch length to elementary amenable groups, see \cite[Theorem 1]{H91}. Note also that elementary amenable groups of type
$FP$ are automatically virtually solvable, see \cite{KMPN09}.
\end{remark}

\section{Concluding remarks.}\label{complete-case}

As mentioned in the introduction, the case of closed {\it complete} affine manifolds provides a large class of 
manifolds satisfying the hypotheses of our {\bf Main Theorem}. For these manifolds, it is tempting to try and give a more direct, geometrical
proof that $||M||=0$. Indeed, one can consider the foliation of $\R^n$ given by affine subspaces in the directions spanned by the
(non-trivial) translational subgroup. Since the developing map is a homeomorphism $D:\widetilde M \rightarrow \R^n$, the translational 
subgroup acts discretely (hence cocompactly) on the leaves of this foliation. Normality of the translational subgroup implies that this 
foliation of $\widetilde M \cong \R^n$ descends to a foliation of $M$ by closed submanifolds, where each leaf is finitely covered by a torus. 
If this foliation was a fibration, then it would follow that $||M||=0$ \cite[Exercise 
14.15]{L02}. More generally, $||M||=0$ if $M$ admits a polarized $\mc{F}$-structure (see 
\cite{CG86}). This geometric approach then motivates the following interesting

\vskip 10pt

\noindent {\bf Question:} If $M$ is a closed aspherical manifold, with a foliation all of whose leaves are finitely covered by tori. Does
it follow that $||M||=0$?

\vskip 10pt

In the non-complete case, one can still foliate the image of the development map $D$ by affine subspaces in the directions spanned by
the translational subgroup, and then pull back this foliation via $D$ to a foliation on $\widetilde M$. The foliation on $\widetilde M$ will still descend
to a foliation on $M$, but it is unclear whether the leaves of the resulting foliation on $M$ are even {\it closed}. Indeed, if one takes a leaf of the 
foliation in $\R^n$, its pre-image in $\widetilde M$ could consist of countably infinitely many leaves for the induced foliation of $\widetilde M$. 
The pre-image of the translational subgroup could then act by permuting
these individual leaves in $\widetilde M$, none of which would close up in $M$.


\end{document}